\documentclass[12pt,reqno,a4paper]{amsart}
\usepackage[utf8]{inputenc}
\usepackage{
    amsmath,  amsfonts, amssymb,  amsthm,   amscd,
    gensymb,  graphicx, comment,  etoolbox, url,
booktabs, stackrel, mathtools,enumitem, mathdots,  microtype, lmodern,    mathrsfs, graphicx, tikz,  longtable,tabularx, float, tikz, pst-node, tikz-cd, multirow, tabularx, amscd,  bm, array, makecell, diagbox, booktabs,ragged2e, caption, subcaption}

\usepackage[all]{xy}
\usepackage{makecell,slashbox}

\usepackage{xcolor}
\usepackage[utf8]{inputenc}
\usepackage{microtype, fullpage, wrapfig,textcomp,mathrsfs,csquotes,fbb}
\usepackage[colorlinks=true, linkcolor=blue, citecolor=blue, urlcolor=blue, breaklinks=true]{hyperref}
\usepackage[capitalise]{cleveref}
\usepackage{todonotes}
\usetikzlibrary{positioning}
\usetikzlibrary{shapes,arrows.meta,calc}

\usetikzlibrary{arrows}

\newtheorem{theorem}{Theorem}[section]

\newtheorem{corollary}[theorem] {Corollary}
\newtheorem{definition}[theorem]{Definition}

\newtheorem{lemma}[theorem]{Lemma}

\newtheorem{proposition}[theorem]{Proposition}
\newtheorem{remark}[theorem]{Remark}

\newcommand\R{\mathbb{R}}
\newcommand\Z{\mathbb{Z}}
\newcommand\C{\mathbb{C}}
\newcommand{\cs}{\mathcal{S}}
\newcommand{\TC}{\mathbf{TC}}

\newcommand{\cl}{\mathbf{cl}}
\newcommand{\zcl}{\mathbf{zcl}}
\newcommand{\cat}{\mathbf{cat}}
\newcommand{\X}{X_{\bullet}}

\newcommand{\id}{\mathrm{id}}
\newcommand{\vr}{\mathrm{VR}}

\def\dist#1#2#3#4{{d}_{#1}^{#2}({#3},{#4})}
\def\VR#1{\mathrm{VR}_{\bullet}({#1})}

\def\nil{\mathrm{nil}}
\def\ra{\rightarrow}

\newcolumntype{x}[1]{>{\centering\arraybackslash}p{#1}}

\title{Equivariant persistence topological complexity}

\author{Navnath Daundkar}
\address{Indian Institute of Technology Madras, Chennai, India}
\email{navnath@iitm.ac.in}
\author[A. Sarkar]{Ankur Sarkar}
\address{Stat Math Unit, Indian Statistical Institute, Kolkata, India}
\email{ankurimsc@gmail.com}
\author{Bittu Singh}
\address{Indian Institute of Technology Bombay, Mumbai, India}
\email{22d0781@iitb.ac.in}

\begin{document}
\begin{abstract}
We introduce equivariant persistent analogues of Lusternik–Schnirelmann category, topological complexity, cup length, and zero-divisors cup length for persistent spaces with group actions, and establish their stability. In particular, we investigate their behavior under equivariant interleavings and obtain bounds for the corresponding erosion distances. For Vietoris–Rips filtrations of compact metric spaces equipped with compatible group actions, we relate the erosion distances between the resulting equivariant persistent invariants to appropriate equivariant versions of the Gromov–Hausdorff distance.

\end{abstract}
\maketitle

\section{Introduction}
A motion planning problem in robotics involves developing an algorithm that takes a pair of configurations of a mechanical system as input and produces a continuous path connecting these two configurations. 
In many applications, the underlying space may depend on a parameter, and it is therefore natural to consider a family of spaces together with maps describing their evolution as the parameter varies. Such a family provides a natural framework for studying the evolution of topological properties or the shape of the data set across different parameter values and is commonly referred to as a persistent space. Persistent homology provides one important approach to extracting and tracking topological information from such parameterized spaces.

Several classical homotopy invariants provide quantitative measures of the  complexity of a space. Among them, the Lusternik--Schnirelmann category \cite{LS_cat} and topological complexity \cite{Farber_TC} are fundamental examples of such invariants, while cup length and zero-divisors cup length provide important cohomological lower bounds for these invariants. In the real world, spaces can be discrete, noisy, and scale-dependent. Recently Contessoto et. all \cite{Persistent_cup_legth} introduced the persistent cup length for the purpose of capturing ring-theoretic information about the evolution of the cohomology ring structure across a filtration and established homotopical stability with respect to homotopy interleaving distance. More generally  M\'emoli et. all introduced Persistent LS category in \cite{MSZ24} and obtained stability result with respect to homotopy-interleaving distance and Gromov-Hausdorﬀ distance. Since the classical topological complexity introduced by Farber \cite{Farber_TC} is defined for a fixed topological space, it does not directly capture the evolution of motion-planning complexity across scales. On the other hand, standard persistent homology primarily records homological information rather than topological complexity. To address this, M\'emoli and Zhou introduced persistent analogues of topological complexity and zero-divisors cup length as a cohomological lower bound in \cite{mavsmoli2025persistence}, thereby providing a framework for studying motion-planning complexity across scales.

This naturally raises the question of how to handle situations in which the underlying data set or topological space possesses symmetries that should be preserved throughout the persistence process. Such symmetries arise naturally in mechanical systems, where configuration spaces often admit group actions representing symmetries of the system. 
In this direction, Fadell \cite{Fadelleqcat} and 
Marzantowicz \cite{Eqlscategory} studied equivariant version of LS category which is later extended for equivariant maps by Garc{\'\i}a-Calcines and Daundkar in \cite{daundkar2025equivariant}.
To incorporate these symmetries into the motion planning problem, Colman and Grant developed equivariant topological complexity in  \cite{Colman_Grant_Eqtc}, which is the minimum number of domains of continuity of the motion planner in the space that preserves the symmetry. More recently, Daundkar, Santhanam and Singh \cite{daundkar2026equivariant} developed these notions in a more gerneral setting, in particular introducing equivariant topological complexity of maps.
However, similar to the classical topological complexity, this notion does not capture the complexity of motion planning when the underlying space varies persistently. Thus, to study persistent motion planning in the presence of symmetries, it is natural to consider a persistent $G$-space. More precisely, a $G$-persistent object in a category $\mathcal{C}$ is a functor $F\colon (P,\leq)\to \mathcal{C}^G$, where $(P,\leq)$ is preorderd set and $\mathcal{C}^G$ is the category of $G$-objects in $\mathcal{C}$ and $G$-equivariant morphism.

For a persistent $G$-space, the structure maps preserve the given group action, allowing equivariant invariants to be studied throughout the persistence process. We introduce the notion of a $G$-categorical invariant (see Definition~\ref{G Categorical Invariant}) and show that the equivariant cup length (Definition \ref{cup_length}), zero-divisors cup length (Definition \ref{zcl}), Lusternik--Schnirelmann category (Definition \ref{LS_cat}), and topological complexity (Definition \ref{per_tc}) satisfy this property (see Propositions~\ref{Categorical invariance} and \ref{Categorical invariance 2}). We also introduce the equivariant erosion distance for interval invariants associated to persistent $G$-spaces (see Definition~\ref{errosion}). This provides a common framework for associating interval invariants to persistent $G$-spaces.
In this setting, we prove that the equivariant erosion distance between the resulting persistent invariants is bounded above by the equivariant interleaving distance (Theorem~\ref{E vs GI}). We further obtain a comparison with the equivariant homotopy-interleaving distance under an appropriate homotopy-invariance assumption (Theorem~\ref{erosion_vs_homotopy_interleaving}). For Vietoris--Rips filtrations, Proposition~\ref{prop: d^G VR leq 2d_GH} and Corollary~\ref{thm: HI between VR leq 2d_GH} provide the corresponding homotopy-interleaving estimate in terms of the equivariant Gromov--Hausdorff distance. As a consequence, for $\mathbf{I}^G=\cl_G,\zcl_G,\cat_G,$ or $\TC_G$, we obtain the corresponding erosion-distance bounds for persistent $G$-spaces and, in the Vietoris--Rips setting, for compact $G$-metric spaces (see Corollary~\ref{corollary: erosion bounds}).

\section{Preliminaries}
In this subsection, we review the results and constructions that will be used in our equivariant persistent setting. We first recall persistent equivariant cohomology for filtered $G$-spaces. We then recall the categorical framework of persistent $G$-objects, $G$-interleavings, and the associated $G$-Gromov--Hausdorff distance. Finally, we review the equivariant versions of topological complexity and Lusternik--Schnirelmann category for equivariant maps, together with the cohomological lower bound given by equivariant zero-divisors cup length.

\subsection{Persistent versions of LS category and topological complexity}

We recall the basic definition and properties of Persistent topological complexity and persistent LS category from \cite{mavsmoli2025persistence}.
\begin{definition}
    Let $X_{\bullet}\colon (\mathbb{R},\leq)\to \mathrm{Top}$ be a persistent space. Then the \emph{persistent topological complexity} of $X_{\bullet}$ is the functor 
   $ \TC(X_{\bullet})\colon (\mathscr{Int},\subseteq)\to (\mathbb{N}\cup \{0\})$, defined by  \[\TC(X_{\bullet})([a,b])=\TC_G(f^b_a:X_a\to X_b).\] Here $\TC(f)$ denote the topological complexity of a map $f$ introduced in \cite{Sc}.
\end{definition}

\begin{definition}
    Let $(\mathcal{P},\leq)$ denote the poset category. A \emph{persistent object} in a category $\mathcal{C}$ is a functor $X\colon \mathcal{P}\to \mathcal{C}$. 

    Thus for each object $a$ of $\mathcal{P}$, we have an object $X_a$ in $\mathcal{C}$ and for every object $a, b$ in $\mathcal{C}$ with $a\leq b$, there is a morphism $f^b_a\colon X_a\to X_b$ such that $f^a_a=\id_{X_a}$ and $f^a_c= f^b_c\circ f^a_b$ for every $a\leq b\leq c$.
\end{definition}
This invariants are closely related to various cohomological notions.
\begin{definition}
The \emph{Lusternik--Schnirelmann category} of a persistent space $X_\bullet\in \mathrm{Top}^{(\mathbb{R},\leq)}$ is the functor
\(\cat(X_\bullet)\colon
(\mathscr{Int},\subseteq)\longrightarrow
(\mathbb{N}\cup\{\infty\},\leq)\), defined by  \[\cat(X_{\bullet})([s,t])= \cat_G(f^s_t\colon X_s\to X_t), 
\]
where $\cat(f)$ denotes the LS category of the map $f$ (see \cite{DK23}).   
\end{definition}
\begin{definition}
    The \emph{persistent zero-divisor cup length} of a persistent space $X_{\bullet}\in \mathrm{Top}^{(\mathbb{R},\leq )}$ is the functor 
    \(\zcl(X_{\bullet}):(\mathscr{Int},\subseteq)\to (\mathbb{N}\cup \{\infty\})\),  defined by  \[\zcl(X_{\bullet})([a,b])=\zcl\big((f^b_a\times f^b_a)^{\ast}:\ker{(\Delta_{X_a})^{\ast}} \to \ker{(\Delta_{X_b})^{\ast}}\big).\]
\end{definition}
\subsection{Equivariant persistent theory}
To incorporate the symmetries arising from a $G$-action into persistent topological data, we first recall the Borel cohomology of $G$-spaces and then describe how it naturally extends to a filtered $G$-space.

We begin by recalling Borel cohomology. Fix a topological group $G$ and let $X$ be a $G$-space and $EG\to BG$ be the universal principal $G$-bundle. Then the Borel construction $X_G^h: =EG\times_G X$ of $X$ is given by the quotient of the space $EG\times X$ under the diagonal $G$-action given by $(e,x). g= (e.g, g^{-1}.x)$.

The Borel equivariant cohomology of $X$ with coefficients in a commutative ring $R$ is defined by
\[H_G^{\ast}(X;R):= H^*(X_G^h;R).\]

If $f\colon X\to Y$ is a $G$-equivariant map, then it induces a map
\[
f_G^h\colon X_G^h\longrightarrow Y_G^h,\qquad
[e,x]\longmapsto [e,f(x)].
\]
Consequently, $f$ induces a homomorphism $f^*\colon H_G^{\ast}(Y;R)\longrightarrow H_G^{\ast}(X;R)$.

A filtraion of $G$-spaces $\{X_r\}_{r\in \mathbb{R}}$ is a functor from $(\mathbb{R},\leq)$ to the category of $G$-spaces. Now applying the cohomology functor $H_G^{\ast}(-,\mathbb{F})$ to this filtered $G$-space, we obtain equivariant cohomology groups $\{H_G^{\ast}(Y_r;\mathbb{F})\}_{r\in \mathbb{R}}$ and morphisms $\{h_{r,s}^*\colon H_G^{\ast}(X_s;\mathbb{F})\to H_G^{\ast}(X_r;\mathbb{F})\}_{r\leq s}$ satisfying 
\[h_{r,r}^*= \id_{H_G^{\ast}(X_r;\mathbb{F})} \text{ and } h_{s,t}^*\circ h_{r,s}^*= h_{r,t}^* \text{ for all } r\leq s\leq t.\]
This whole data is termed as a persistent equivariant cohomology module associated to the filtered $G$-space $\{X_r\}_{r\in \mathbb{R}}$. For further background on persistent equivariant cohomology, we refer the reader to \cite{adams2024persistent, MR2854319, MR2121296}.

We now recall the categorical framework to study persistence in the equivariant setting. We begin with $G$-persistent objects, $G$-interleaving distance, together with the associated $G$-Gromov-Hausdorff distance.
\begin{definition}
    \emph{Translation} on a preordered set $(P,\leq)$ is a functor $\Gamma\colon P\to P$ with a natural transformation $\eta_{P} \colon I_P\Rightarrow \Gamma$.
\end{definition}
Let $\textbf{Trans}_P$ denote the monoid of translations.

\begin{definition}
    Let $G$ be a group and $\mathcal{C}$ be a category. We define $\mathcal{C}^G: = Func(BG,\mathcal{C})$, where $BG$ is the one-object category whose automorphism group is $G$.
    
    Then clearly the object of $\mathcal{C}^G$ is an object $X\in \mathcal{C}$ with a $G$-action $\rho\colon G\to \mathrm{Aut}(X)$ and morphisms are $G$-equivariant morphism.
\end{definition}
\begin{definition}
    A \emph{persistent $G$-object} in a category $\mathcal{C}$ is a functor $F\colon (P,\leq)\to \mathcal{C}^G$, where $(P,\leq)$ is preorderd set.
\end{definition}

\begin{definition}
    Let $P$ be a preordered set and let $\Gamma, K\in \textbf{Trans}_P$. Suppose $F,G\in \mathcal{C}^G$. Then a $G$-$(\Gamma,K)$-interleaving between $F$ and $G$ consists of natural transformations $\phi\colon F\Rightarrow G\Gamma$ and $\psi\colon G\Rightarrow FK$ such that $(\psi\Gamma)\phi= F\eta_{K\Gamma}$ and $(\phi K)\psi=G \eta_{\Gamma K}$.

    We say that $F, G$ are \emph{$G$-$(\Gamma,K)$-interleaved} if there exists a $G$-$(\Gamma,K)$-interleaving between them. 
    For brevity, a $G$-$(\Gamma,\Gamma)$-interleaving is simply termed $(\Gamma,G)$-interleaved.
\end{definition}
\begin{definition}
    We call a translation $\Gamma$ an \emph{$\epsilon$-translation} if $\omega_{\Gamma}\leq \epsilon$, where $\omega\colon \textbf{Trans}_P\to [0,\infty)$ is sublinear projection (see \cite{BDS}).

    $F,G \in \mathcal{C}^G$ are \emph{$G$-$\epsilon$-interleaved} with respect to $\omega$ if there exist $\epsilon$-translations $\Gamma, K\in \textbf{Trans}_P$ such that $F$ and $G$ are $G$-$(\Gamma,K)$-interleaved.
\end{definition}

\begin{definition}
    The \emph{$G$-interleaving distance} between $F, G\in \mathcal{C}^G$ is 
    \[d_{I}^{G\omega}(F,G):=\mathrm{inf}\{\epsilon\mid F,G \text{ are } G\text{-}\epsilon\text{-interleaved with respect to }\omega\}.\]
\end{definition}
\begin{definition}[{\cite{LM25}}]

The $G$-Gromov-Hausdorff distance between $G$-metric spaces $(X,d_X^G)$ and $(Y,d_Y^G)$ is defined by $$d_{GH}^G(X,Y)=\frac{1}{2}\inf_{C\in \Gamma_G(X,Y)} \sup_{(x,y),(x',y')\in C} |d_X^G(x,x')-d_Y^G(y,y')|,$$
    where $\Gamma_G(X,Y)$ denotes the set of $G$-equivariant correspondence between $X$ and $Y$, that is, the set consisting all $G$-invariant subsets of $X\times Y$.
    \end{definition}

\subsection{Equivariant versions of topological complexity and LS category of equivariant maps}
The equivariant topological complexity and equivariant LS category of equivariant maps were introduced in \cite{daundkar2026equivariant} and \cite{daundkar2025equivariant}, respectively. We recall their definitions and the properties that will be used in this article.

\begin{definition}
    The equivariant topological complexity of a $G$-map $f\colon X\to Y$, denoted by $\TC_G(f)$, is the smallest non-negative integer $n$ for which there is a $G$-invariant open cover $\{U_0,\dots,U_n\}$ of $X\times X$, such that for each $i=0,\dots,n$, there is a $G$-map $s_i\colon U_i\to Y^I$ satisfying $s_i(x_0,x_1)(0)=f(x_0)$ and $s_i(x_0,x_1)(1)=f(x_1)$. If no such $n$ exists, we define $\TC_G(f)=\infty$.
\end{definition}

\begin{definition}
Let $f\colon X \to Y$ be a $G$-map.  
The \emph{equivariant LS category} of $f$, denoted $\cat_G(f)$, is the least integer $n\geq 0$ such that $X$ admits a cover $\{U_i\}_{i=0}^n$ by $n+1$ $G$-invariant open subsets with the property that, for each $i$, the restriction $f|_{U_i}$ is $G$-homotopic to a $G$-map $U_i\to Y$ whose image lies in the orbit of some $y_i\in Y$.  
If no such $n$ exists, we set $\cat_G(f)=\infty$.
\end{definition}

The topological complexity of a map admits a useful cohomological
lower bound in terms of the zero-divisor cup-length. For maps, this
cohomological invariant was introduced by Scott in \cite{Sc}. In the
equivariant setting, its analogue was introduced in
\cite{daundkar2026equivariant}, providing a corresponding lower bound
for the equivariant topological complexity. We recall its definition
below.
\begin{definition}
    The $G$-equivariant zero divisor cup length of a $G$-map $f\colon X\to Y$, denoted by $\zcl_G(f)$, is the largest non-negative integer $n$ for which there exists cohomology classes $\alpha_1,\dots,\alpha_k\in H_G^{\ast}(Y\times Y)$ satisfying $\Delta_Y^*(\alpha_i)=0$ for each $i=0,\dots,n$ and $(f\times f)^*(\alpha_1\cup \alpha_2\cdots\cup \alpha)\neq 0$. Here $\Delta_Y\colon Y\to Y\times Y$ denotes the diagonal map. In other words,
    \[\zcl_G(f):= \nil\big(\mathrm{im}(\ker(\Delta_Y)_G^{h*}\xrightarrow{(f\times f)^*} \ker (\Delta_X)_G^{h*})\big).\]
\end{definition}

Now, we recall some properties $\TC_G(f)$, which will be useful later.
\begin{lemma}[{\cite{daundkar2026equivariant}}]\label{properties_of_tc_of_maps}
Let $f, f_1, f_2\colon X\to Y$ and $g\colon Y\to Z$ be $G$-maps.
    \begin{enumerate}
        \item $\TC_G(\id_X)=\TC_G(X)$.
        \item $\TC_G(f)\leq \mathrm{min}\big\{ \TC_G(X), \TC_G(Y)\big\}$.
        \item Let $H$ and $H'$ be closed subgroups of $G$ such that $X^H, Y^H$ are $H'$-invariant. Then $\TC_{H'}(f^H)\leq \TC_G(f)$.
        \item $\cat_H(f)\leq \TC_G(f)$, if $H$ is a stablizer of some point in $X$.
        \item $\TC_G(f)\leq \cat_G(f\times f)$, if $Y$ is $G$-connected.
        \item $\TC_G(g\circ f)\leq \mathrm{min}\big\{\TC_G(f),\TC_G(g)\big\}$.
        \item $\TC_G(f_1)=\TC_G(f_2)$, if $f_1$ is $G$-homotopic to $f_2$.
    \end{enumerate}
\end{lemma}

\section{Equivariant persistence LS category}
To include symmetry in persistent topological complexity, it is natural to consider an equivariant version of classical categorical invariants. In this section, we introduce the $G$-persistent Lusternik--Schnirelmann category of a $G$-persistent space. We then compute this invariant for several natural examples of $G$-persistent spaces.
\begin{definition}\label{LS_cat}
Let $X_\bullet\colon (\mathbb{R}_{\geq 0},\leq)\to \mathcal{C}^G$ be a persistent $G$-space. 
The \emph{$G$-persistent Lusternik--Schnirelmann category} of $X_\bullet$ is the functor
\(\cat_G(X_\bullet)\colon
(\mathscr{Int},\subseteq)\longrightarrow
(\mathbb{N}\cup\{\infty\},\leq)\), defined by  \[\cat_G(X_{\bullet})([s,t])= \cat_G(f^s_t\colon X_s\to X_t), 
\]
where $\cat_G(f)$ denotes the equivariant LS category of the $G$-map $f$, introduced in \cite[Definition 4.2]{daundkar2025equivariant}.
\end{definition}

\begin{definition}
Let $f:X\ra Y$ be a $G$-map and $R$ be a commutative ring. Then the \emph{$G$-cup length} of $f$ id defined by ${\cl}_G(f;R):=\nil((f)_G^{h\ast}(\tilde H_G^\ast(Y))$.
\end{definition}

\begin{definition}\label{cup_length}
    Let $X_\bullet$ be a persistent $G$-space.
A \emph{$G$-persistent zero-divisor-cup-length} is the functor 
\(\cl_G(X_{\bullet};R)\colon (\mathscr{Int},\subseteq)\to (\mathbb{N}\cup \{\infty\}), \) defined by  \[\cl_G(X_{\bullet};R)([a,b])=\cl_G(f_a^b;R).\]
\end{definition}

\begin{definition}\label{G Categorical Invariant}
Let $\mathcal C^G$ be the category of $G$-objects in $\mathcal C$. A \emph{$G$-categorical invariant} is a map $\mathbf{I}^G: \operatorname{Ob}(\mathcal C^G)\sqcup \operatorname{Mor}(\mathcal C^G) \longrightarrow (\mathbb R_{\geq 0},\leq)$ such that 
\begin{enumerate}
    \item $\mathbf{I}^G(\id_X)=\mathbf{I}^G(X)$ for every
    $X\in\operatorname{Ob}(\mathcal C^G)$;
    
    \item for any composable $G$-equivariant maps $X\xrightarrow{f}Y\xrightarrow{g}Z\xrightarrow{h}W$
    in $\mathcal C^G$, one has
    \[
    \mathbf{I}^G(h\circ g\circ f)\leq \mathbf{I}^G(g).
    \]
\end{enumerate}
\end{definition}
\begin{lemma}\label{lem: composition ineq}
    The $G$-cup length $\cl_G(-)$ satisfies
    \begin{enumerate}
        \item $\cl_G(\id_X)=\cl_G(X)$.
        \item If $f\colon X\to Y, g\colon Y\to Z$ and $h\colon Z\to W$ be $G$-maps, then $\cl_{G}(h\circ g\circ f)\leq \cl_G(g)$.
    \end{enumerate}
\end{lemma}
\begin{proof}
    We note that $(\id_X)_G^{h*}(\tilde{H}_G^*(X;R))= \tilde{H}_G^*(X;R)$. Hence $\cl_G(\id_X;R)=\nil(\tilde{H}_G^*(X;R))=\cl_G(X)$.

    Suppose $X\xrightarrow{f}Y\xrightarrow{g}Z\xrightarrow{h} W$ be $G$-equivariant maps and $\cl_G(g;R)=k$. Then $(h\circ g\circ f)_G^{h*}= f_G^{h*}\circ g_G^{h*}\circ h_G^{h*}$. Hence $(h\circ g\circ f)_G^{h*}(\tilde{H}_G^{h*}(W;R))$. Let $u_1,\ldots,u_{k+1} \in (h\circ g\circ f)_G^{h*} (\tilde H_G^*(W;R))$. Hence, for each $i$, there exists $a_i\in g_G^{h*}(\tilde H_G^*(Z;R))$ such that $u_i=f_G^{h*}(a_i)$ and so $u_1\cup\cdots\cup u_{k+1}= f_G^{h*}(a_1\cup\cdots\cup a_{k+1})$. Since $\cl_G(g;R)=k$, we have $a_1\cup\cdots\cup a_{k+1}=0$. Thus $u_1\cup\cdots\cup u_{k+1}=0$. This gives $\cl_G(h\circ g\circ f;R) \leq k.$
\end{proof}

\begin{lemma}\label{lem: composition ineq for cat}
 Let $f: X\to Y$ and $g:Y\to Z$ be $G$-maps between $G$-spaces. Then 
 $$\cat_G(g\circ f)\leq min\{\cat_G(f),\cat_G(g)\}.$$
\end{lemma}
\begin{proof}
 It is easy to see that $\cat_G(\mathrm{id}_X)=\cat_G(X)$. Now we show that $\cat_G(g\circ f)\leq min\{\cat_G(f),\cat_G(g)\}$. The inequality $\cat_G(g\circ f)\leq \cat_G(f)$ is obvious. To prove $\cat_G(g\circ f)\leq \cat_G(g)$, consider a $G$-invariant open set $V\subseteq Y$ such that $g|_{V}\simeq_G h$, where $h$ is a $G$-map taking values in an orbit of some point. Let $U=f^{-1}(V)$. Then $(g\circ f)|_{U}=g|_{V}\circ f|_{U}\simeq_G h\circ f|_{U}$. Note that since $h$ takes values in some orbit, $h\circ f|_{U}$ also takes values in the same orbit. This proves the desired inequality.    
\end{proof}

\begin{proposition}\label{Categorical invariance}
   $\cl_G,\cat_G\colon\mathrm{Ob}(\mathrm{Top}^G)\sqcup \mathrm{Mor}(\mathrm{Top}^G)\ra \mathbb N\cup\{\infty\}$ are $G$-categorical invariants. 
\end{proposition}
\begin{proof}
 Follows from Lemma \ref{lem: composition ineq} and Lemma \ref{lem: composition ineq for cat}.
\end{proof}

We now compute these invariants for several standard $G$-spaces
and group actions. We first consider examples arising from
Vietoris--Rips filtrations.

\begin{remark}
    Let $(X,d)$ be a metric space equipped with an action of a group $G$ by isometries. For $r>0$, the Vietoris-Rips complex $\vr_{\leq r}(X)$ admits a natural simplicial $G$-action defined on simplices by
\[
g\cdot\{x_0,\ldots,x_k\} = \{g x_0,\ldots,g x_k\}, \text{ for all }g\in G.
\]
Moreover, $d(g\,x, g\,y)=g(x,y)$ for all $x,y\in X$ and $g\in G$. 
Thus $\vr_{\bullet}(X)$ is naturally a $G$-persistent space.
\end{remark}

We first compute the equivariant persistent LS category for the reflection action on a sphere.
\begin{proposition}
    Suppose $\mathbb{Z}_2$ acts on $\mathbb{S}^n$ by reflection in the last coordinate. Then 
    \[\cat_{\mathbb{Z}_2}(\vr_{\bullet}(\mathbb{S}^n))(J)=
    \begin{cases}
    1,& \text{if } J\subset  \big(0,\mathrm{arccos}(-\frac{1}{n+1})\big),\\
        0, & \text{if } J\not\subset [0,\pi].
    \end{cases}\]
\end{proposition}
\begin{proof}
By \cite[Theorem 8]{LM25}, $\vr_t(\mathbb{S}^n)\simeq_{\mathbb{Z}_2} \mathbb{S}^n$ for $0<t<\arccos\left(-\frac{1}{n+1}\right)$. Moreover, $\mathrm{cat}_{\Z_2}(S^n)=1$ (see \cite[Example 5.9]{Colman_Grant_Eqtc}). Then the result follows. 
\end{proof}

For the complex projective space with the conjugation action, we have
the following.
\begin{lemma}\label{Cat_cal_3}
    Let $\mathbb{Z}_2$ acts on $\mathbb{CP}^n$ by complex conjugation. Then $\cat_{\mathbb{Z}_2}(\mathbb{CP}^n)=n$.
\end{lemma}
\begin{proof}
    We note that $(\mathbb{CP}^n)^{\mathbb{Z}_2}=\mathbb{RP}^n$. This implies that $\cat_{\mathbb{Z}_2}(\mathbb{CP}^n)\geq \cat(\mathbb{RP}^n)=n$.

    Next consider the open sets $U_j=\{[z_0,\dots,z_n]\in \mathbb{CP}^n\mid z_j\neq 0\}$ for $0\leq j\leq n$. Clearly, $\mathbb{CP}^n=\bigcup_{j=0}^n U_j$.

    Let $\phi\colon U_j\to \mathbb{C}^n$ be defined as
    \[\phi_j([z_0,\dots,z_n]):=\big(\frac{z_0}{z_j},\dots, \frac{z_{j-1}}{z_j},\frac{z_{j+1}}{z_j},\dots,\frac{z_n}{z_j}\big).\]

    One can show $\phi_j$ is a homeomorphism with inverse given by \[(w_1,\dots,w_n)\mapsto [w_1,\dots,w_{j-1},1,w_{j+1},\dots, w_n].\]
    Note that under this homeomorphism $p_j=[0,\dots,0,1,0\dots,0]$ is mapped to $(0,\dots,0)$.

    Now consider the homotopy $H_j\colon \mathbb{C}^n\times I\to \mathbb{C}^n$ defined by $$H(w_1,\dots,w_n,t)=\big((1-t)w_1,\dots,(1-t)w_n\big).$$ Note that $H(\bar{w}_1,\dots,\bar{w}_n,t)= \overline{H(w_1,\dots,w_n,t)}$. This gives that the inclusion $i_j\colon U_j\hookrightarrow \mathbb{CP}^n$ is $\mathbb{Z}_2$-homotopic to a map $c_j\colon U_j\to \mathbb{CP}^n$ given by $c_j(x)=p_j$. Also, the image of $c_j$ is contained in the single orbit of $\{p_j\}$. This gives that each $U_j$ is $\mathbb{Z}_2$-categorical. Hence $\cat_{\mathbb{Z}_2}(\mathbb{CP}^n)\leq n$. This completes the proof.
\end{proof}

\begin{remark}
\Cref{Cat_cal_3} is a special case of \cite[Theorem 3.2]{BDS26}, where the result is proved for quasitoric manifolds from a different perspective. $\C P^n$ is an example of a quasitoric manifold. Here we have given a direct proof.
\end{remark}
If $\mathbb{Z}_2$ acts on $\mathbb{S}^n$ by reflection on the last co-ordinate, then $\tilde{H}_{\mathbb{Z}_2}^*(\mathbb{S}^n;\mathbb{Q})=0$. Hence $\cl_{\mathbb{Z}_2}(\mathbb{S}^n;\mathbb{Q})=0$.

\section{Equivariant persistence topological complexity}
We now introduce $G$-persistent topological complexity as the equivariant analogue of persistent topological complexity introduced in \cite{mavsmoli2025persistence} and also show that it is an equivariant categorical invariant. We then compute this invariant for the
Vietoris-Rips filtration of a sphere equipped with a reflection
action. To obtain its lower bound, we introduce another equivariant categorical invariant, namely $G$-persistent zero-divisor-cup-length and establish its relation with $G$-persistent topological complexity. We conclude with several computations of the $G$-persistent zero-divisor-cup-length for some static persistent $G$-spaces.

\begin{definition}\label{per_tc}
Let $X_\bullet$ be a persistent $G$-space.
A \emph{$G$-persistent topological complexity} is the functor 
\(\TC_G(X_{\bullet})\colon (\mathscr{Int},\subseteq)\to (\mathbb{N}\cup \{0\})\),  defined by  \[\TC_G(X_{\bullet})([a,b])=\TC_G(f^b_a:X_a\to X_b).\]
\end{definition}
Now combining Lemma~\ref{properties_of_tc_of_maps}{(3)} with the above definition of equivariant persistent topological complexity, we obtain the following proposition.

\begin{proposition}
Let $X_\bullet$ be a persistent $G$-space, and let $H,H'$ be closed
subgroups of $G$. If $X_a^H,Y_a^H$ are $H'$-invariant for every $a\in\mathbb{R}$,
then
\[
\TC_{H'}(X_\bullet^H)([a,b])
\leq
\TC_G(X_\bullet)([a,b])
\]
for every $[a,b]\in\mathscr{Int}$.
\end{proposition}

\begin{proposition}\label{TC_cal}
    Suppose $\mathbb Z_2$ acts on $\mathbb{S}^n$ by reflection with respect to the last coordinate. Then 
    \[\TC_{Z_2}(\vr_{\bullet}(\mathbb{S}^n))(J)=\begin{cases}
     3, & \text{if } n\geq 2 \text{ and } J\subset \big(0,\mathrm{arccos}(-\frac{1}{n+1})\big),\\
        0, & \text{if } J\not\subset [0,\pi].
    \end{cases}\]
\end{proposition}
\begin{proof}
By \cite[Theorem 8]{LM25}, $\vr_t(\mathbb{S}^n)\simeq_{\mathbb{Z}_2} \mathbb{S}^n$ for $0<t<\arccos\left(-\frac{1}{n+1}\right)$. 
On the other hand, for $t>\pi$, every pair of points in $\mathbb{S}^n$
has geodesic distance at most $\pi<t$. Hence $\vr_t(\mathbb{S}^n)$ is the abstract simplex on the vertex set $\mathbb{S}^n$, and therefore $\vr_t(\mathbb{S}^n)\simeq *$. Consequently, if $t\not \in [0,\pi]$, then
$\vr_t(\mathbb{S}^n)$ is contractible.
Combining these observations with
$\TC_{\mathbb{Z}_2}(\mathbb{S}^n)= 3$
for $n\geq 3$ from \cite[Example 5.9]{Colman_Grant_Eqtc}, gives the desired result.
\end{proof}

\begin{definition}\label{zcl}
    Let $X_\bullet$ be a persistent $G$-space.
A \emph{$G$-persistent zero-divisor-cup-length} is the functor 
\(\zcl_G(X_{\bullet};R)\colon (\mathscr{Int},\subseteq)\to (\mathbb{N}\cup \{\infty\}), \) defined by  \[\zcl_G(X_{\bullet};R)([a,b])=\zcl_G(f_a^b;R).\]
\end{definition}
\begin{proposition}\label{Categorical invariance 2}
   $\zcl_G,\TC_G\colon \mathrm{Ob}(\mathrm{Top}^G)\sqcup \mathrm{Mor}(\mathrm{Top}^G)\ra \mathbb N\cup\{\infty\}$ are $G$-categorical invariants. 
\end{proposition}
\begin{proof}
    For $\mathbf{I}^G= \zcl_G, \TC_G$, we know that $\mathbf{I}^G(\id_X)=\mathbf{I}^G(X)$ for all $G$-spaces $X$.
    Moreover, $\mathbf{I}^G(f\circ g)\leq \min\{\mathbf{I}^G(f),\mathbf{I}^G(g)\}$ for all $G$-maps $X\xrightarrow{g}Y\xrightarrow{f}Z$ (see \cite{daundkar2026equivariant}).
\end{proof}
Thus we have categorical invariants associated to a persistent $G$-space $X_\bullet$.
\begin{proposition}
Let $X_{\bullet}$ be a persistent $G$-space. Then
$$\cl_G(\X)\leq \cat_G(\X)\leq \TC_G(\X).$$ 
\end{proposition}
\begin{proof}
The proof follows from the following inequalities
    $$\cl_G(X_\bullet)[a,b]=\cl_G(f_a^b)\leq \cat_G(f_a^b)\leq \TC_G(f_a^b).$$
\end{proof}

From \cite[Theorem 3.15]{daundkar2026equivariant}, we have the following lemma.
\begin{lemma}\label{lowerbound}
Let $f\colon X\to Y$ be a $G$-map. Suppose $u_i\in H_G^*(X\times X;R)$ satisfy $u_i\in
\ker((\Delta_X)_G^{h*})\cap \mathrm{im}\bigl((f\times f)_G^{h*}\bigr)$ for $0\leq i\leq k$, where $R$ is a commutative ring. If $u_0\cup \dots\cup u_k\neq 0$, then $k+1\leq \TC_G(f)$.   
\end{lemma}

Using Lemma~\ref{lowerbound}, we obtain the following.
\begin{proposition}
    Let $X_{\bullet}$ be a persistent $G$-space. Then 
    \[\zcl_G(X_{\bullet};R)\leq \TC_G(X_{\bullet}).\]
\end{proposition}

Now we provide some computations of $\zcl_G$ in the static setting. 
\begin{lemma}\label{zcl_cal}
    Let $\mathbb{Z}_2$ act on $\mathbb{S}^n$ by reflection in the last coordinate. Then 
    \[\zcl_{\mathbb{Z}_2}(\mathbb{S}^n;\mathbb{Q})=1.\]
\end{lemma}
\begin{proof}

Suppose the action of $\mathbb{Z}_2$ on $\mathbb{S}^n$ is given by the map $\alpha\colon \mathbb{S}^n\to S^n$, defined as $$\alpha(x_1,\dots,x_n)=(x_1,\dots,x_{n-1},-x_n).$$ Then if $u$ be the generator of $H^n(\mathbb{S}^n;\mathbb{Q})$, then $\alpha^*(u)=-u$. This implies that $H_{\mathbb{Z}_2}^*(\mathbb{S}^n;\mathbb{Q})=H^*(\mathbb{S}^n;\mathbb{Q})^{\mathbb{Z}_2}=\mathbb{Q}$.

Now we analyze $H_{\mathbb{Z}_2}(\mathbb{S}^n\times \mathbb{S}^n;\mathbb{Q})= H^*(\mathbb{S}^n\times \mathbb{S}^n;\mathbb{Q})^{\mathbb{Z}_2}$. Using Kunneth formula, we note that $u_1=1\otimes u$ and $u_2=u\otimes 1$ are the generators of $H^n(\mathbb{S}^n\times \mathbb{S}^n;\mathbb{Q})$ and $u_1 \cup u_2= u\otimes u$ is the generator of $H^{2n}(\mathbb{S}^n\times \mathbb{S}^n;\mathbb{Q})$. Clearly $\alpha^*(u_1)=-u_1$ and $\alpha^*(u_2)=-u_2$. Therefore $\alpha^*(u_1\cup u_2)= u_1 \cup u_2$. Thus $H^*(\mathbb{S}^n\times \mathbb{S}^n;\mathbb{Q})^{\mathbb{Z}_2}\cong \mathbb{Q}\oplus \mathbb{Q}\{u_1\cup u_2\}$. 

It follows from $\Delta_{\mathbb{S}^n}^*(u_1)=u= \Delta_{\mathbb{S}^n}^*(u_2)$ that the homomorphism  $$(\Delta_{\mathbb{S}^n})_{\mathbb{Z}_2}^{h\ast}\colon H_{\mathbb{Z}_2}^*(\mathbb{S}^n\times \mathbb{S}^n;\mathbb{Q})\to H_{\mathbb{Z}_2}^*(\mathbb{S}^n;\mathbb{Q})$$ satisfies $(\Delta_{\mathbb{S}^n})_{\mathbb{Z}_2}^{h\ast}(u_1\cup u_2)= u^2=0$. Hence $u_1\cup u_2\in \ker((\Delta_{\mathbb{S}^n})_{\mathbb{Z}_2}^{h\ast})$ is the only non-zero class, so $\zcl_{\mathbb{Z}_2}(\mathbb{S}^n)= 1$. 
\end{proof}

\begin{lemma}
    Suppose $\mathbb{S}^1$ acts on $X=\mathbb{S}^1$ by rotation. Then
    \[\zcl_{\mathbb{S}^1}(\mathbb{S}^1;\mathbb{Q})=1.\]
\end{lemma}

\begin{proof}
 Here the action is free and transitive, so $H_{\mathbb{S}^1}^*(X;\mathbb{Q})\cong H^*(X/{\mathbb{S}^1};\mathbb{Q})\cong \mathbb{Q}$. 
 Also $H_{\mathbb{S}^1}^*(X\times X;\mathbb{Q})\cong H^*(X\times X/{\mathbb{S}^1};\mathbb{Q})\cong H^*(\mathbb{S}^1;\mathbb{Q})\cong \mathbb{Q}\oplus \mathbb{Q}\{u\}$, where $u$ is the generator of $H_{\mathbb{S}^1}^1(X\times X;\mathbb{Q})$. This gives that $(\Delta_X)_{\mathbb{S}^1}^{h\ast}$ is trivial. Hence $u\in \ker((\Delta_X)_{\mathbb{S}^1}^{h\ast})$. As $u^2=0$, we have $\zcl_{\mathbb{S}^1}(\mathbb{S}^1;\mathbb{Q})=1$.
\end{proof}

\begin{lemma}
    Let $\mathbb{Z}_2$ acts on $\mathbb{CP}^n$ by complex conjugation. Then $\zcl_{\mathbb{Z}_2}(\mathbb{CP}^n;\mathbb{Q})\geq n$.
\end{lemma}

\begin{proof}
    Let $\tau\colon \mathbb{CP}^n\to \mathbb{CP}^n$ given by $\tau([z_1,\dots,z_n])=[\bar{z}_1,\dots,\bar{z}_n]$ represents the $\mathbb{Z}_2$-action on $\mathbb{CP}^n$.

    If $u$ be the generator of $H^2(\mathbb{CP}^n;\mathbb{Q})$, then we have $\tau^*(u^k)=(-1)^k u^k$. This implies that 
    \[H_{\mathbb{Z}_2}^*(\mathbb{CP}^n;\mathbb{Q})\cong H^*(\mathbb{CP}^n;\mathbb{Q})^{\mathbb{Z}_2}\cong 
    \begin{cases}
      \mathbb{Q}\oplus \mathbb{Q}\{u^2\}\oplus \dots\oplus\mathbb{Q}\{u^{2n}\}, \text{ if } n \text{ is even},\\
       \mathbb{Q}\oplus \mathbb{Q}\{u^2\}\oplus \dots\oplus\mathbb{Q}\{u^{2n-2}\}, \text{ if } n \text{ is odd}.
    \end{cases}
    \]
    Similarly, we have \[H_{\mathbb{Z}_2}^*(\mathbb{CP}^n\times \mathbb{CP}^n;\mathbb{Q})\cong \big(\frac{\mathbb{Q}[u,v]}{u^{n+1}, v^{n+1}}\big)^{\mathbb{Z}_2}.\]
    Note that $u^kv^l\in H_{\mathbb{Z}_2}^*(\mathbb{CP}^n\times \mathbb{CP}^n;\mathbb{Q})$ if either $k$ and $l$ are both odd or $k$ and $l$ are both even. Also $u^{2i}, v^{2i}\in H_{\mathbb{Z}_2}^*(\mathbb{CP}^n\times \mathbb{CP}^n;\mathbb{Q})$ for $i=1,\dots,\lfloor \frac{n}{2} \rfloor+1$.

    Since $\Delta_X^*\colon H^*(\mathbb{CP}^n\times \mathbb{CP}^n;\mathbb{Q})\to H^*(\mathbb{CP}^n;\mathbb{Q})$ maps both $u$ and $v$ to $u$, we have $uv-u^2\in \ker(\Delta_X)_{\mathbb{Z}_2}^{h*}$. As $(uv-u^2)^n\neq 0$ in $H_{\mathbb{Z}_2}^*(\mathbb{CP}^n\times \mathbb{CP}^n;\mathbb{Q})$, $\zcl_{\mathbb{Z}_2}(\mathbb{CP}^n;\mathbb{Q})\geq n$.
\end{proof}

\begin{proposition}
    Suppose $\mathbb{Z}_2$ acts on $\mathbb{S}^n$ by reflection in the last coordinate. Then 
    \[\zcl_{\mathbb{Z}_2}(\vr_{\bullet}(\mathbb{S}^n))(J)=
    \begin{cases}
    1,& \text{if } J\subset  \big(0,\mathrm{arccos}(-\frac{1}{n+1})\big),\\
        0, & \text{if } J\not\subset [0,\pi].
    \end{cases}\]
\end{proposition}
\begin{proof}
Using a similar argument as in Proposition~\ref{TC_cal}, we obtain the result from Lemma~\ref{zcl_cal}.   
\end{proof}
\section{Stability}
In this section, we show that the equivariant persistent invariant is a stable persistent invariant. We first introduce the notions and auxiliary results needed for the stability argument. We follow the techniques developed by Blumberg and Lesnick \cite{BL} to deduce the stability principle in the equivariant setting.

We first define the equivariant analogue of the homotopy interleaving distance introduced in \cite{BL}.

Let $X_{\bullet}, Y_{\bullet}\colon (\mathbb{R},\leq)\to Top^G$ be persistent $G$-spaces. A natural transformation $\eta\colon X_{\bullet}\Rightarrow Y_{\bullet}$ is called objectwise \emph{$G$-weak equivalence} if each component $\eta_t\colon X_t\to Y_t$ is a $G$-weak homotopy equivalence (i.e., $\eta_t^H\colon X_t^H\to Y_t^H$ is a weak homotopy equivalence for each subgroup $H$ of $G$).

We say that $X_{\bullet}$ and $Y_{\bullet}$ are \emph{$G$-weakly equivalent} if there exists a $G$-persistent space $W_{\bullet}$ together with objectwise $G$-weak equivalences $W_{\bullet}\to X_{\bullet}$ and $W_{\bullet}\to Y_{\bullet}$. In this case, we write $X_{\bullet}\simeq^G Y_{\bullet}$.

For $\delta\geq 0$, two persistent spaces $X_{\bullet}$ and $Y_{\bullet}$ are said to be \emph{$G$-$\delta$-homotopy interleaved} if there exist $G$-weakly persistent spaces $X_{\bullet}'\simeq^G X_{\bullet}$ and $Y_{\bullet}'\simeq^G Y_{\bullet}$ such that $X_{\bullet}'$ and $Y_{\bullet}'$ are $G$-$\delta$-interleaved.  

The equivariant analogue of the homotopy interleaving distance was recently introduced by Halder, Sau, and Sen in \cite{debashis2026equivariant_persistent}. Independently and concurrently, we were developing a related notion in the context considered in this paper, without knowledge of their work at the time. 
\begin{definition}
    The \emph{$G$-homotopy interleaving distance} $d_{HI}^G$ between persistent $G$-spaces $X_{\bullet}$ and $Y_{\bullet}$ is defined as follows:
    \[d_{HI}^G(X_{\bullet},Y_{\bullet}):=\mathrm{inf}\{d_{I}^{Top^G}(X_{\bullet}',Y_{\bullet}')\mid X_{\bullet}'\simeq^G X_{\bullet}, Y_{\bullet}'\simeq^G Y_{\bullet}\}.\]
\end{definition}

\begin{definition}
    Given a $G$-space $T$ and a $G$-equivariant function (not necessarily continuous) $\gamma:T\to \mathbb R$, we can define the \emph{sublevel filtration} $\mathcal S(\gamma):\mathbb R\to \mathrm{Top}^G$ by 
    $$\mathcal S(\gamma)_r=\gamma^{-1}(-\infty,r].$$
     Clearly, $\mathcal S(\gamma)_r$ is a $G$-invariant subset of $T$, since $\gamma$ is a $G$-function.
\end{definition}
For two function $\gamma, \gamma':T\to \mathbb R$, we define $$d_\infty(\gamma, \gamma')=\sup_{x\in T}|\gamma(x)-\gamma'(x)|.$$

By a distance, we mean an extended pseudometric. 
\begin{definition}
    A  $G$-invariant distance $d^G$ on $G$-spaces is
    \begin{enumerate}
        \item \emph{stable} if for all $G$-spaces $X$ and $G$-functions $\gamma,\gamma':X\to \R$,
        $$d^G(\mathcal S(\gamma), \mathcal S(\gamma'))\leq d_{\infty}(\gamma,\gamma'),$$
        \item \emph{homotopy invariant} if $d^G(X,Y)=0$ for all $X\simeq^G Y.$
    \end{enumerate}
\end{definition}

Since $d_I^{\omega}(F,G)\leq d_{I}^{G\omega}(F,G)$ for all $F,G\in \mathcal{C}^G$ \cite[Proposition 4.4.8]{G-interleaving_distance}, we obtain the following proposition from \cite[Theorem 1]{MSZ24}.
\begin{remark}
    Let $\mathbf{I}$ be a categorical invariant. Then, we have the following inequality for any $G$-persistent object $X_\bullet, Y_{\bullet}\colon (\mathbb R,\leq)\ra \mathcal C^G$
    \[d_{E}(\mathbf{I}(X_{\bullet}),\mathbf{I}(Y_{\bullet}))\leq d_{E}^G(\mathbf{I}(X_{\bullet}),\mathbf{I}(Y_{\bullet}))\leq d_{I}^G(X_{\bullet},Y_{\bullet}),\] where $d_E$ denotes the errosion distance.
\end{remark}

From the definition, we note that 
\[d_{HI}(X_{\bullet},Y_{\bullet})\leq d_{HI}^G(X_{\bullet},Y_{\bullet}).\] 
It now follows from \cite[Theorem 2]{MSZ24} that 
\begin{remark}
    Let $\mathbf{I}$ be a categorical invariant of topological spaces such that it is invariant under pre- and post-composition with weak homotopy equivalences. Then for any persistent $G$-spaces $X_{\bullet}$ and $Y_{\bullet}$, we have 
    \[d_E(\mathbf{I}(X_{\bullet}), \mathbf{I}(Y_{\bullet}))\leq d_{HI}^G(X_{\bullet}, Y_{\bullet})\] and
    \[d_E(\mathbf{I}(\vr_{\bullet}(X), \mathbf{I}(\vr_{\bullet}(Y)))\leq 2 \cdot d_{GH}^G(X,Y).\]
\end{remark}

Let $X_{\bullet}\colon (\mathbb{R},\leq)\to \mathcal{C}^G$ be a persistent $G$-space. For $a\leq b$, suppose $X_{ab}\colon X_a\to X_b$ denote the corresponding $G$-equivariant structure map. The $G$-persistence $\mathbf{I}^G$-invariant is the function $\mathbf{I}_X^G\colon \mathscr{Int}\to \mathbb{R}_{\geq 0}$ defined by $\mathbf{I}_X^G([a,b]):=\mathbf{I}^G(X_{ab})$.

\begin{definition}\label{errosion}
Let $X_{\bullet}, Y_{\bullet}$ be $G$-persistent spaces. Let
$\mathbf{I}_X^G,\mathbf{I}_Y^G: (\mathscr{Int},\subseteq)\to (\mathbb{R}_{\geq0},\leq)$
be interval invariants associated to $X_{\bullet}$ and
$Y_{\bullet}$, respectively. For $\epsilon>0$, we call $\mathbf{I}_X^G$ and $\mathbf{I}_Y^G$ are \emph{$\epsilon$-erroded} if for every $[a,b]\in \mathscr{Int}$, 
\[\mathbf{I}_X^G([a,b])\geq \mathbf{I}_Y^G([a-\epsilon, b+\epsilon])\text{ and } \mathbf{I}_Y^G([a,b])\geq \mathbf{I}_X^G([a-\epsilon,b+\epsilon]).\]

The $G$-equivariant erosion distance between $X_{\bullet}$ and $Y_{\bullet}$ is defined by
\[d_E^G(X_{\bullet}, Y_{\bullet}):= \mathrm{inf}\big\{\epsilon>0\mid \mathbf{I}_X^G\text{ and }\mathbf{I}_Y^G \text{ are $\varepsilon$-eroded} \big\}.,\]
 with the convension that $d_E^G(X_{\bullet}, Y_{\bullet})=\infty$ if no such $\epsilon$ exists.
\end{definition}

The following theorem shows that the $G$-categorical invariant is $1$-Lipschitz stable and later we will strengthen this result by replacing equivariant interleaving distance by equivariant homotopy interleaving distance.

\begin{theorem}\label{E vs GI}
    Let $\mathbf{I}^G\colon \operatorname{Ob}(\mathcal C^G)\sqcup \operatorname{Mor}(\mathcal C^G) \longrightarrow (\mathbb R_{\geq 0},\leq)$ be $G$-equivariant categorical invariant of $\mathcal{C}^G$. Then for any persistent $G$-spaces $X_{\bullet}$ and $Y_{\bullet}$ in $\mathcal{C}$, we have 
    \[d_E^G(\mathbf{I}^G(X_{\bullet}), \mathbf{I}^G(Y_{\bullet}))\leq d_{I}^G(X_{\bullet}, Y_{\bullet}).\]
\end{theorem}
\begin{proof}
The proof is identical to that of \cite[Theorem 1]{MSZ24}, with the ordinary categorical invariant and morphisms replaced by their $G$-equivariant counterparts in $\mathcal C^G$.
\end{proof}
The following lemma shows that $G$-categorical invariants are compatible with the $G$-weak equivalences appearing in the definition of $G$-$\delta$-homotopy interleavings and will be used in the proof of the stability result.
\begin{lemma}\label{I^G_property}
    Let $\mathbf{I}^G$ be a $G$-categorical invariant such that for any $G$-equivariant maps $X\xrightarrow{f} Y\xrightarrow{g} Z\xrightarrow{h}W$ in $\mathcal{C}^G$, if $g$ is a $G$-weak homotopy equivalence, then $\mathbf{I}^G(g\circ f)= \mathbf{I}^G(f)$ and $\mathbf{I}^G(h\circ g)=\mathbf{I}^G(h)$. If the $G$-persistent spaces $X_{\bullet}$ and $X_{\bullet}'$ are such that $X_{\bullet}\simeq^G X_{\bullet}'$, then $\mathbf{I}^G(X_{\bullet})=\mathbf{I}^G(X_{\bullet}')$.
\end{lemma}
\begin{proof}
    Since $X_{\bullet}\simeq^G X'{\bullet}$, there exists a $G$-persistent space $W_{\bullet}$ and objectwise $G$-weak equivalences $\phi\colon W_{\bullet}\Rightarrow X_{\bullet}$ and $\psi\colon W_{\bullet}\Rightarrow X_{\bullet}'$. This implies that for each $t\in \mathbb{R}$, the maps $\phi_t\colon W_t\to X_t$ and $\psi_t\colon W_t\to X_t'$ are $G$-weak homotopy equivalence. Using this, it follows from the following commutative diagram
\[\begin{tikzcd}
	{Z_t} & {Z_s} \\
	{X_t} & {X_s}
	\arrow["{g_t^s}", from=1-1, to=1-2]
	\arrow["{\phi_t}"', "\simeq",from=1-1, to=2-1]
	\arrow["{\phi_s}"', "\simeq",from=1-2, to=2-2]
	\arrow["{f_t^s}"', from=2-1, to=2-2]
\end{tikzcd}\]
    that   \[\begin{aligned}
\mathbf{I}_X^G([t,s])
&=\mathbf{I}^G(f_s^t) =\mathbf{I}^G(f_s^t\circ \phi_t) =\mathbf{I}^G(\phi_s\circ g_s^t) =\mathbf{I}^G(g_s^t) =\mathbf{I}_W^G([t,s]).
\end{aligned}\]
Hence $\mathbf{I}^G(X_{\bullet})= \mathbf{I}^G(W_{\bullet})$. Similarly, we can show that $I(X_{\bullet}')= \mathbf{I}^G(W_{\bullet})$. Consequently, $\mathbf{I}^G(X_{\bullet})= \mathbf{I}^G(X_{\bullet}')$.
\end{proof}

\begin{proposition}\label{prop: d^G VR leq 2d_GH}
    A stable and $G$-homootopy invariant distance $d^G$ on $G$-equivariant metric spaces satisfies $$d^G(\VR X,\VR Y)\leq2\cdot d_{GH}^G(X,Y).$$
\end{proposition}
\begin{proof}
    Let $X, Y$ be two metric $G$-spaces with $\dist{GH}{G}{X}{Y}<\delta.$
    Then there exists $G$-invariant subset $C\subseteq X\times Y$ with $|d_X^G(x,x')-d_Y^G(y,y')|\leq2\delta,$ for all $(x,y), (x',y')\in C.$
    Let $[C]$ be the full simplicial complex formed by the vertex set $C.$ We note that $[C]$ is a $G$-space. We also have the $G$-equivariant projection maps $p_X\colon C\to X$ and $p_Y\colon C\to Y.$

    We define a filtration $F_\bullet^X$ on $[C]$ where $\sigma\in F^X_r$ if and only if $\dist{X}{G}{p_X(u)}{p_X(v)}\leq 2r$ for all $u,v \in \sigma.$
    We note that for all $gu,gv \in g\sigma$,
    \begin{align*}
        d_X^G(p_X(gu),p_X(gv))&= d_X^G(gp_X(u),gp_X(v)) & \text{ (since the projection } p_X \text{ is } G\text{-equivariant)}\\
        &= d_X^G(p_X(u),p_X(v)) & \text{ (since the metric } d_X^G \text{ is } G\text{-equivariant)}\\
        &\leq 2r.
    \end{align*}
    Thus $F_\bullet^X$ is a $G$-simplicial filtration.
    This induces a $G$-function $\gamma^X\colon [C]\to \R$ given by $$\gamma^X(\sigma)=\min\{r\in \R\mid \sigma\in F_r^X\}.$$ Therefore $F_\bullet^X=\mathcal S(\gamma^X)$ is the $G$-simplicial sublevel filtration.
    Similarly, we define a filtration $F_\bullet^Y$ so that $F_\bullet^Y=\mathcal S(\gamma^Y).$

     We have $\dist{\infty}{}{\gamma^X}{\gamma^Y}=\sup_{\sigma\in [C]} |\gamma^X(\sigma)-\gamma^Y(\sigma)|\leq \delta.$
     By stability assumption for $d^G$, we obtain
     $$d^G(F_\bullet^X,F_\bullet^Y)=d^G(\mathcal S(\gamma^X),\mathcal S(\gamma^Y))\leq d_\infty(\gamma^X,\gamma^Y)\leq \delta.$$
     We note that $p_X:C\to X$ induces a $G$-equivariant simplicial map $g_r:F^X_{r/2}\to \vr_r(X)$. 
     The equivariant Quillen's Theorem A \cite[Theorem A.2]{Bergner_eqTree} says that if $f:S\to T$ is a $G$-simplicial map of $G$-simplicial complexes such that $(f^H)^{-1}(\sigma)$ is contractible for all subgroups $H\leq G$, then $f$ is a $G$-homotopy equivalence. By Quillen's Theorem A,  we conclude $g_r$ is a $G$-homotopy equivalence for all $r\geq 0.$ Thus $g: F_{\bullet/2}^X\to \VR X$ is an objectwise $G$-homotopy equivalence.
     Similarly, by considering the map $F^Y_{r}\to \vr_r(Y)$ induced from $p_Y:C\to Y$, we obtain a $G$-homotopy equivalence  $F_{\bullet/2}^Y\to \VR Y$.
     Using the $G$-homootopy invariance assumption for $d^G$, we have $d^G(\VR X ,F_{\bullet/2}^X)=0$ and $d^G(\VR Y,F_{\bullet/2}^Y)=0.$

     Using the triangle inequality, this yields
     $$d^G(\VR X, \VR Y)\leq d^G(\VR X ,F_{\bullet/2}^X)+ d^G(F_{\bullet/2}^X,F_{\bullet/2}^Y) +d^G(\VR Y,F_{\bullet/2}^Y)\leq2\delta.$$
     Since the inequality holds for every $\delta>d_{GH}^G(X,Y)$, passing through the infimum, we get $d^G(\VR X, \VR Y)\leq 2d_{GH}^G(X,Y).$ 
\end{proof}
\begin{theorem}\label{thm: stability invariance}
$d_{HI}^G$ is a distance on $G$-spaces satisfies: \begin{enumerate}
    \item stability,
    \item $G$-homotopy invariance.
\end{enumerate}
\end{theorem}
\begin{proof}
It is clear from the definition of $d_{HI}^G$ is non-negative, symmetric and $d_{HI}^G(X,X)=0$ for any $G$-space $X$. The triangle inequality follows from similar arguments in \cite[section 4.2]{BL}.

    (1) Let $\gamma, \gamma'\colon X\to \R$ be two $G$-maps with $d_\infty(\gamma,\gamma')\leq \epsilon.$
    We have the inclusion of sublevel sets
    $$\cs(\gamma)_r=\{x\in X|\gamma(t)\leq r\}\subseteq\{x\in X|\gamma'(t)\leq r+\epsilon\}=\cs(\gamma')_{r+\epsilon}.$$
    We denote this inclusions by $\phi_r\colon \cs(\gamma)_r\to \cs(\gamma')_{r+\epsilon}.$
Symmetrically, we get inclusions $\psi_r:\cs(\gamma')_r\to \cs(\gamma)_{r+\epsilon}.$
The compositions $$\psi_{r+\epsilon}\circ\phi_r=\cs(\gamma)(r \leq r+2\epsilon)$$ and $$\phi_{r+\epsilon}\circ\psi_r=\cs(\gamma')(r \leq r+2\epsilon).$$ Consequently, $d_I^G(\cs(\gamma),\cs(\gamma'))\leq \epsilon.$
Therefore $d_{HI}^G(\cs(\gamma),\cs(\gamma'))\leq d_I^G(\cs(\gamma),\cs(\gamma'))\leq \epsilon.$
Thus $d_{HI}^G(\mathcal S(\gamma), \mathcal S(\gamma'))\leq d_\infty(\gamma, \gamma').
$

(2) For two $G$-spaces $X, Y$ with $X\simeq_G Y$, we have $d_{HI}^G(X,Y)\leq d_I(X,X)=0$ by the definition of homotopy interleaving distance.
\end{proof}
As a consequence of Proposition~\ref{prop: d^G VR leq 2d_GH} and Theorem~\ref{thm: stability invariance}, we establish the following.
\begin{corollary}\label{thm: HI between VR leq 2d_GH}
For $G$-metric spaces $X$ and $Y$, 
     $$d_{HI}^G(\VR X,\VR Y)\leq2\cdot d_{GH}^G(X,Y).$$
\end{corollary}

The following theorem is an equivariant analogue of \cite[Theorem 2]{MSZ24} which is strengthening the Theorem~\ref{E vs GI}. We also establish the stability of $G$-categorical invariants with respect to the equivariant Gromov--Hausdorff distance, building on Corollary~\ref{thm: HI between VR leq 2d_GH}.

\begin{theorem}\label{erosion_vs_homotopy_interleaving}
 Let $\mathbf{I}^G$ be a $G$-categorical invariant of topological spaces such that $\mathbf{I}^G$ is invariant under pre- and post-composition with $G$-weak homotopy equivalences. Then for persistent $G$-spaces $X_{\bullet}$ and $Y_{\bullet}$, we have 
    \begin{equation}
        d_E^G(\mathbf{I}^G(X_{\bullet}), \mathbf{I}^G(Y_{\bullet}))\leq d_{HI}^G(X_{\bullet}, Y_{\bullet}).
    \end{equation}

    For compact $G$-metric spaces $(X,d_X)$ and $(Y,d_Y)$, we have 
    \begin{equation}\label{eq2}
        d_E^G(\mathbf{I}^G(\vr_{\bullet}(X)), \mathbf{I}^G(\vr_{\bullet}(Y)))\leq 2\cdot d_{GH}^G(X,Y).
    \end{equation}
    \end{theorem}
\begin{proof}
 Let $X_{\bullet}', Y_{\bullet}'$ be two persistent $G$-spaces such that $X_{\bullet}\simeq^G X_{\bullet}'$ and $Y_{\bullet}\simeq^{G} Y_{\bullet}'$. Then by Lemma~\ref{I^G_property}, we have $d_E^G(\mathbf{I}^{G}(X_{\bullet}),\mathbf{I}^{G}(Y_{\bullet}))=d_E^G(\mathbf{I}^{G}(X_{\bullet}'),\mathbf{I}^{G}(Y_{\bullet}'))$. Hence it follows from Theorem~\ref{E vs GI} that $d_E^G(\mathbf{I}^{G}(X_{\bullet}), \mathbf{I}^{G}(Y_{\bullet}))\leq d_{I}^G(X_{\bullet}',Y_{\bullet}')$. Therefore, by the definition of the homotopy $G$-interleaving distance $d_E^G(\mathbf{I}^G(X_{\bullet}), \mathbf{I}^G(Y_{\bullet}))\leq d_{HI}^G(X_{\bullet}, Y_{\bullet})$.

    Combining this inequality with Corollary~\ref{thm: HI between VR leq 2d_GH}, we obtain \[d_E^G(\mathbf{I}^G(\vr_{\bullet}(X)), \mathbf{I}^G(\vr_{\bullet}(Y)))\leq 2\cdot d_{GH}^G(X,Y).\]
\end{proof}

\begin{corollary}\label{corollary: erosion bounds}
Let $\mathbf{I}^G=\cl_G$, $\zcl_G, \cat_G$ or $\TC_G$.
    Suppose $X_{\bullet}$ and $Y_{\bullet}$ be two persistence $G$-spaces with $G$-CW complex structures. Then 
    \[d_{E}^G(\mathbf{I}^G(X_{\bullet}),\mathbf{I}^G(Y_{\bullet}))\leq d_{HI}^G(X_{\bullet},Y_{\bullet}).\]

    Moreover, for compact $G$-metric spaces $X$ and $Y$, we have 
    \[d_E^G(\mathbf{I}^G(\vr_{\bullet}(X)),\mathbf{I}^G(\vr_{\bullet}(Y)))\leq 2\cdot d_{GH}^G(X,Y).\]
\end{corollary}
\begin{proof}
    Recall the equivariant Whitehead's theorem (see \cite{Bredon1967EquivariantCohomology}) if $f\colon X\to Y$ is a $G$-weak homotopy equivalence between $G$-spaces admitting a $G$-CW structure, then $f$ is a $G$-homotopy equivalence. Moreover, $\cl_G, \cat_G, \zcl_G$ and $\TC_G$ are invariant under post and pre-composition with $G$-homotopy equivalence by Propositions~\ref{Categorical invariance} and \ref{Categorical invariance 2}. Now the desired inequality follows from Theorem~\ref{erosion_vs_homotopy_interleaving}.

    Since Vietoris–Rips filtrations of compact metric spaces with $G$-action are persistent $G$-CW complexes, the second inequality is a consequence of~\eqref{eq2}.
\end{proof}

\begin{proposition}\label{lower_bound}
    Let $b>0$ and let $a_x,a_y$ be positive real numbers such that $\tfrac{b}{2}< a_y< a_x< b$. Suppose $X_{\bullet}$ and $Y_{\bullet}$ be $G$-persistence spaces with associated $G$-categorical invariants satisfy
    \[\mathbf{I}_X^G(J)=\begin{cases}
        n& \text{if } J\subset (0,a_x),\\
        0& \text{if } J\not\subset [0,b],
    \end{cases}
    \quad \mathbf{I}_Y^G(J)=\begin{cases}
        m & \text{if } J\subset (0,a_y),\\
        0&\text{if } J\not\subset [0,b],
    \end{cases}\]
    where $n> m>0$.
    Then $d_E^G(X_{\bullet}, Y_{\bullet})\geq \tfrac{a_x}{2}$.

    Moreover, if $\mathbf{I}_X^G(J)=0$ whenever the length of $J$ exceeds $a_x$ and $\mathbf{I}_Y^G(J)=0$ whenever the length of $J$ is grater than $a_y$, then $d_E^G(X_{\bullet}, Y_{\bullet})= \tfrac{a_x}{2}$.
\end{proposition}

\begin{proof}
    The proof is identical to that of \cite[Proposition 4.6]{mavsmoli2025persistence}, with $\mathbf{I}_X$ and $\mathbf{I}_Y$ replaced by $\mathbf{I}_X^G$ and $\mathbf{I}_Y^G$, respectively. 
\end{proof}

\begin{corollary}
Let $n\geq 1$. Suppose $\mathbb{Z}_2$ acts on $\mathbb{S}^n$ by reflection in the last coordinate. Then
\[d_E^{\mathbb{Z}_2}(\TC_{\mathbb{Z}_2}(\vr_{\bullet}(\mathbb{S}^1), \TC_{\mathbb{Z}_2}(\vr_{\bullet}(\mathbb{S}^2))))\geq \frac{\alpha_1}{2}.\]
\end{corollary}

\vspace{0.3cm}
\noindent \textbf{Acknowledgement:}
The authors would like to thank the organisers of the  OneMath World School 2026, held at Chennai Mathematical Institute for their valuable discussions.
Navnath Daundkar gratefully acknowledges the support of the DST-INSPIRE Faculty Fellowship (Faculty Registration No.~IFA24-MA218), Department of Science and Technology, Government of India; the Industrial Consultancy and Sponsored Research (IC\&SR), Indian Institute of Technology Madras, through the New Faculty Initiation Grant (RF25261395MANFIG009294).  Bittu Singh is grateful to the Prime Minister's Research Fellowship (PMRF ID 1302644), Government of India, for his financial support.
\vspace{0.5cm}

\noindent \textbf{AI Declaration:}
During the preparation of this manuscript, the authors used ChatGPT (OpenAI) for language editing, including improving grammar, clarity, and readability.

\bibliographystyle{plain} 
\bibliography{ref_eqtc}
\end{document}